\documentclass{proc-l}
\usepackage{amssymb}
\usepackage{xcolor}
\usepackage{hyperref}       % hyperlinks
\usepackage{url}  

\newtheorem{theorem}{Theorem}[section]
\newtheorem*{theorem*}{Theorem}
\newtheorem{lemma}[theorem]{Lemma}
\newtheorem{corollary}[theorem]{Corollary}
\theoremstyle{definition}
\newtheorem{definition}[theorem]{Definition}

\theoremstyle{remark}
\newtheorem{remark}[theorem]{Remark}

\numberwithin{equation}{section}

\begin{document}

\title[Electrostatic system with divergence-free Bach tensor]{Electrostatic system with divergence-free Bach tensor and non-null cosmological constant}


%    author one information
\author{Benedito Leandro}
\address{Universidade Federal de Goiás\\IME\\ Caixa postal 131, CEP 74690-900,  Goi\^ania, GO, Brazil.
}
\curraddr{}
\email{bleandroneto@ufg.br}
\thanks{The authors were partially supported by CNPq Grant 403349/2021-4.}

%    author two information
\author{Róbson Lousa}
\address{Universidade Federal de Goiás\\IME\\ Caixa postal 131, CEP 74690-900,  Goi\^ania, GO, Brazil.}\address{Instituto Federal de Roraima \\Câmpus Amajari\\ CEP 69343-000, Amajari, RR, Brazil.}
\curraddr{}
\email{robson.lousa@ifrr.edu.br}
\thanks{Róbson Lousa was partially supported by PROPG-CAPES [Finance Code 001].}

%    \subjclass is required.
\subjclass[2020]{83C22, 83C05, 53C18.}

\date{}

\dedicatory{}

%    Abstract is required.
\begin{abstract}
We prove that three-dimensional electrostatic manifolds with divergence-free Bach tensor are locally conformally flat, provide that the electric field and the gradient of the lapse function are linearly dependent. Consequently, a three-dimensional electrostatic manifold admits a local warped product structure with a one-dimensional base and a constant curvature surface fiber.
\end{abstract}

\maketitle

\section{Introduction and main results}

 In this paper, we will consider the following system (cf. \cite{cederbaum2016uniqueness, chrusciel2005non, chrusciel2017non, tiarlos, kunduri2018, lucietti} and the references therein). 

\begin{definition}\label{def1}
	Let $(M^3,g)$ be a Riemannian manifold with $E$ a tangent vector field on $M$ and $f\in C^\infty(M)$ satisfying 
	%\begin{equation}\label{eq027}
	\begin{equation}\label{s1}
		\begin{array}{rcll}
			%\label{s1}
			\nabla^2f&=&f(\textnormal{Ric}-\Lambda g+2E^\flat\otimes E^\flat-|E|^2g),\\\\
			\Delta f&=&(|E|^2-\Lambda)f,\quad0=\textnormal{div}E\quad\mbox{and}\quad 0\,=\,\textnormal{curl}(fE).
		\end{array}
	\end{equation}
	Here, $\textnormal{Ric}$, $\nabla^2$, $\textnormal{div}$ and $\Delta$ stand for the Ricci tensor, Hessian tensor, divergence and Laplacian operator with respect to the metric $g$, respectively. Moreover, $E^\flat$ is the one-form metrically dual to $E$. We refer to the above equations as {\it electrostatic  system} with cosmological constant $\Lambda$ for the electrostatic spacetime
	associated to $(M^3, g, f, E)$. 
\end{definition}

Remember, the $\textnormal{curl}$ stands for an operator that describes the circulation (or rotation) of a vector field. Thus, we have $\textnormal{curl}(fE)=0$ if and only if
\begin{eqnarray}\label{curlXld}
df\wedge E^\flat+fdE^\flat=0.
\end{eqnarray}

The smooth function $f$ is called the lapse function, the field $E$ is known as electric field and $M^3$ is the spatial factor for the static electrostatic spacetime. Moreover, $f > 0$ on $M$. If $M$ has boundary $\partial M$, we assume in addition  that $f^{-1}(0)= \partial M$ (cf. \cite{chrusciel2005non, chrusciel2017non,tiarlos, kunduri2018}).

Note that taking the contraction of the first equation and combining it with the Laplacian of $f$ in \eqref{s1}, we obtain an useful equation that relates the scalar curvature $R$, the cosmological constant, and the electric field:
\begin{equation}\label{rrr}
	R=2(|E|^2+\Lambda).
\end{equation}
Furthermore, the system \eqref{s1} implies that the electric field and the gradient of the lapse function are linearly dependent on $\partial M=f^{-1}(0)$ (see \cite[Lemma 4]{tiarlos}).

There are some well-known examples of solutions for the electrostatic system, and we recommend seeing \cite[Section 3]{tiarlos} for a good overview. For instance, the {charged} Nariai system is an $3$-dimensional space $\left[0,\,\frac{\pi}{\alpha}\right]\times\mathbb{S}^2$ with metric tensor $g=dr^2 + \varphi^2g_{\mathbb{S}^2},$ where $\varphi$ is a constant and $g_{\mathbb{S}^2}$ is the standard metric of the sphere $\mathbb{S}^2$ of radius $1$. The electric field and the lapse function are given by
\begin{eqnarray*}
	E=\frac{q}{\varphi^2}\partial_r\quad\mbox{and}\quad f(r(x)) = \sin(\alpha r(x)),
\end{eqnarray*}
where $r(x)^2=x_1^2+x_2^2+x_3^2$ such that $(x_1,\,x_2,\,x_3)$ are Cartesian coordinates, $\alpha=\sqrt{\Lambda - \frac{q^2}{\varphi^4}}$ and $\frac{1}{2\Lambda}<\varphi^2<\frac{1}{\Lambda}.$ Moreover, $0<m^2=\frac{1}{18\Lambda}\left[1+12q^2\Lambda+\sqrt{(1-4q^2\Lambda)^3}\right]$ and $0<|q|\leq\varphi^2\sqrt{\Lambda}.$ It is important to point out that the {charged} Nariai system is locally conformally flat (see \cite[Corollary 1.34]{chow2006hamilton}). In this work, is also important to remember the cold black hole system and the ultracold black hole system. They also are locally conformally flat standard electrostatic models.

The cold black hole is an $3$-dimensional space $[0,\,\infty)\times\mathbb{S}^2$ with metric tensor $g=dr^2 + \varphi^2g_{\mathbb{S}^2},$ where $\varphi$ is a constant and $g_{\mathbb{S}^2}$ is the standard metric of the sphere $\mathbb{S}^2$ of radius $1$. The electric field and the lapse function are given by
\begin{eqnarray*}
	E=\frac{q}{\varphi^2}\partial_r\quad\mbox{and}\quad f(r(x)) = \sinh(\beta r(x)),
\end{eqnarray*}
where $r(x)^2=x_1^2+x_2^2+x_3^2$, $\beta=\sqrt{\frac{q^2}{\varphi^4}-\Lambda}$ and $0<\varphi^2<\frac{1}{2\Lambda}.$ Moreover, $0<m^2=\frac{1}{18\Lambda}\left[1+12q^2\Lambda+\sqrt{(1-4q^2\Lambda)^3}\right]$ and $\varphi^2\sqrt{\Lambda}\leq|q|.$ 

The ultracold black hole is an $3$-dimensional space $[0,\,\infty)\times\mathbb{S}^2$ with metric tensor $g=dr^2 + \varphi^2g_{\mathbb{S}^2},$ where $\varphi^2=\frac{1}{4\Lambda}=q^2$. The electric field and the lapse function are given by
\begin{eqnarray*}
	E=\sqrt{\Lambda}\partial_r\quad\mbox{and}\quad f(r) = r,
\end{eqnarray*}
 Moreover, $m=\frac{1}{3}\sqrt{\frac{2}{\Lambda}}.$

The Reissner-Nordström-de Sitter (RNdS) manifold is a important electrostatic system $(M^3,g,f,E)$ where $$M=[r_+,r_c]\times\mathbb{S}^2$$
to some positive constants $r_+$ and $r_c$ which are solutions of the lapse function given by
$$f=\left(1-\dfrac{2m}{r}+\dfrac{q^2}{r^2}-\dfrac{\Lambda r^2}{3}\right)^{\frac{1}{2}}.$$ Also, it is possible to extend $[r_+,r_c]$ to the entire real line.
The metric and the electric field are given by
$$g=f(r)^{-2}dr^2+r^2g_{\mathbb{S}^2}, \quad E=\dfrac{q}{r^2}f(r)\partial r.$$
In above, $q$, $m$ and $g_{\mathbb{S}^2}$ stand for the charge, mass and the standard metric of unit sphere $\mathbb{S}^2$, respectively. Since the cosmological constant $\Lambda$ is positive, from \eqref{rrr} we can see that $R>0$. 

The RNdS solution can be rewritten in cosmological
coordinates. For instance, the Kastor-Traschen solution represents a $N$ charge-equal-to-mass, i.e., $m=|q|$, black holes in a spacetime with a positive cosmological constant $\Lambda:$ 
$$ds^2=-W^{-2}dt^2 + W^2(dx_{1}^2+dx_{2}^2+dx_{3}^2),$$
where $W=-\sqrt{\frac{\Lambda}{3}}t + \displaystyle\sum_{i=1}^N\frac{m_{i}}{r_{i}}.$ Here, $m_i$ stands for the black hole masses. Moreover, $r_i(x)=\sqrt{(x_1-a_i)^2+(x_2-b_i)^2+(x_3-c_i)^2}$ is the distance from a fixed point $(a_i,\,b_i,\,c_i).$ It is interesting to point out that this solution is time-dependent and correspond to the Majumdar-Papapetrou solution when $\Lambda=0.$

Keeping the electrostatic solutions in mind, we know that the electric field and the lapse function are related. In fact, from \eqref{curlXld} the electric field and the gradient of the lapse function must be linearly dependent at the boundary $\partial M$.

There are some well-known classification results of some important geometric structures like static vacuum manifolds and Ricci solitons carrying a metric such that the Bach tensor is free from divergence (cf. \cite{cao, catino2017gradient, hwang2021vacuum, benedito, qing2013note}). Any three-dimensional a Riemannian manifold is locally conformally flat if, and only if, its Cotton tensor $C$ is identically zero.

In the three dimensional case the Cotton tensor is associated with the Bach tensor, $B$, accordingly to $B=\textnormal{div}C$. The Bach tensor was defined in 1921 by Rudolf Bach and it is connected to general relativity and conformal geometry. This tensor appeared naturally from studies of  Huyghens's principle and has some psychical significance mainly about wave propagation (see for instance \cite{szekeres1968conformal} and the references therein).

The main goal of this work is to show that an electrostatic system with divergence-free Bach tensor, i.e., $\textnormal{div}^2B=0$,  must be locally conformally flat. It is important to say that $\textnormal{div}^2B=0$ is less restrictive (topologically speaking) than asymptotically flat conditions.

To state our main results we need to define an important function. To that end, we will say that the electric field $E$ and the gradient of the lapse function $\nabla f$ are linearly dependent if there exists a smooth function $\rho$ such that $E=\rho\nabla f.$ As an interest consequence of $\textnormal{curl}(fE)=0$, if $E=\nabla \psi$ for some smooth function $\psi:M\to\mathbb{R}$, then $E$ must be parallel to the gradient of $f$. Also, it is natural to consider the case $fE=\nabla\psi$.

We define the function $$Q=2(1-f^2\rho^2).$$
\begin{theorem}\label{teoharmonico'}
	Let $(M^3,\,g,\,f,\,E)$ be a compact (without boundary) electrostatic system such that the electric field and the gradient of the lapse function are linearly dependent. Suppose that the Bach tensor is divergence-free and $Q>0$ (or $Q<0$). Then, $(M^3,\,g)$ is locally conformally flat. 
\end{theorem}

The next result proves the noncompact case.
\begin{theorem}\label{proper'}
	Let $(M^3,\,g,\,f,\,E)$ be an electrostatic system such that the electric field and the gradient of the lapse function are linearly dependent. Suppose that the Bach tensor is divergence-free and $Q>0$ (or $Q<0$). If $f$ is a proper function, then $(M^3,\,g)$ is locally conformally flat. 
\end{theorem}

Now, we are able to provide the geometric structure for an $3$-dimensional electrostatic system.
\begin{theorem}
	\label{fiber007-3}
Let $(M^3,\,g,\,f,\,E)$ be an electrostatic system such that the electric field and the gradient of the lapse function are linearly dependent. Suppose that the Bach tensor is divergence-free and $Q>0$ (or $Q<0$). If $f$ is a proper function, around any regular point of $f$ the manifold is locally a warped product with a one-dimensional base with fiber $(N^2,\,\overline{g})$ of constant curvature, i.e.,
	$$(M^{3},\,g)=(I,\,dr^{2})\times_{\varphi}(N^{2},\,\overline{g}),$$
	where $I\subset\mathbb{R}$ and $\varphi(r)=c_1\int\frac{dr}{\sqrt{f(r)}}+c_2;$ $c_1$ and $c_2$ are constants. 
\end{theorem}

\begin{remark}
	It is important to point out that if $M^3$ is compact in Theorem \ref{fiber007-3}, it is not necessary to ask for $f$ to be a proper function.
\end{remark}

\section{Structural lemmas}\label{lemmas}

This section is reserved to some preliminary results to prove the main theorems of this work. We start constructing a covariant $V$-tensor similar to the tensor defined in \cite{andrade}.

To that end, first we combine \eqref{s1} with \eqref{rrr} to obtain
\begin{equation}\label{combinado}
	\nabla^2f=f\left(\textnormal{Ric}+2E^\flat\otimes E^\flat-\dfrac{R}{2}g\right).
\end{equation}

On other hand, it is well known that in any Riemannian manifold we can relate the Riemannian curvature tensor with a smooth function by using the Ricci identity
\begin{equation*}\label{4}
	\nabla_i\nabla_j\nabla_kf-\nabla_j\nabla_i\nabla_kf=R_{ijkl}\nabla^lf.
\end{equation*}
Since the Hessian operator is symmetric, taking the covariant derivative of \eqref{combinado} over $i$ and $j$ and then subtract them we get

\begin{eqnarray*}
	R_{ijkl}\nabla^lf&=&\nabla_i\nabla_j\nabla_kf-\nabla_j\nabla_i\nabla_kf\\&=&f(\nabla_iR_{jk}-\nabla_jR_{ik})-\frac{f}{2}(\nabla_iRg_{jk}-\nabla_jRg_{ik})\\&&-\frac{R}{2}(\nabla_ifg_{jk}-\nabla_jfg_{ik})+(R_{jk}\nabla_if-R_{ik}\nabla_jf)\\
	&&+2f(E^{\flat_j}\nabla_iE^{\flat_k}-E^{\flat_i}\nabla_jE^{\flat_k}+\nabla_iE^{\flat_j}E^{\flat_k}-\nabla_jE^{\flat_i}E^{\flat_k})\\
	&&+2(\nabla_ifE^{\flat_j}E^{\flat_k}-\nabla_jfE^{\flat_i}E^{\flat_k}),
\end{eqnarray*}
Here, we are considering  $\{e_i\}^{3}_{i=1}$ as a base for the tangent space of $M$. Moreover, $E^{\flat_i}=E^{\flat}(e_i)$.
Note that the Cotton tensor over a $3$-dimensional Riemannian manifold is defined by
\begin{equation}\label{ct}
	C_{ijk}=\nabla_iR_{jk}-\nabla_jR_{ik}-\frac{1}{4}(\nabla_iRg_{jk}-\nabla_jRg_{ik}).
\end{equation}
Furthermore, the Riemann curvature tensor is given by
\begin{eqnarray*}%\label{wt}
	R_{ijkl}&=&R_{ik}g_{jl}-R_{il}g_{jk}+R_{jl}g_{ik}-R_{jk}g_{il}-\frac{R}{2}(g_{ik}g_{jl}-g_{il}g_{jk}).
\end{eqnarray*}
Therefore, combining these equations we get
\begin{eqnarray}\label{construção1}
	fC_{ijk}&=&(R_{jl}\nabla^lfg_{ik}-R_{il}\nabla^lfg_{jk})+R(\nabla_ifg_{jk}-\nabla_jfg_{ik})+2(R_{ik}\nabla_jf-R_{jk}\nabla_if)\nonumber\\
	&&-2f(E^{\flat_j}\nabla_iE^{\flat_k}-E^{\flat_i}\nabla_jE^{\flat_k}+\nabla_iE^{\flat_j}E^{\flat_k}-\nabla_jE^{\flat_i}E^{\flat_k})\\
	&&-2E^{\flat_k}(E^{\flat_j}\nabla_if-E^{\flat_i}\nabla_jf)+\dfrac{f}{4}(\nabla_iRg_{jk}-\nabla_jRg_{ik}).\nonumber
\end{eqnarray}

Now, using $\textnormal{curl}(fE)=0$ we can infer that
\begin{eqnarray*}
	fdE^{\flat}(e_i,\,e_j)=-(df\wedge E^\flat)(e_i,\,e_j)&=& E^{\flat}(e_i)df(e_j)-E^{\flat}(e_j)df(e_i)\\
	&=&E^{\flat_i}\nabla_jf - E^{\flat_j}\nabla_if.
\end{eqnarray*}
On the other hand, let $dE^{\flat}(e_i,\,e_j)=E^{\flat_{ij}}$, then, by definition we have
\begin{eqnarray}\label{2forma}
	E^{\flat_{ij}}=\nabla_iE^{\flat_j} - \nabla_jE^{\flat_i}.
\end{eqnarray}
Further, we can see that
\begin{eqnarray}\label{curlajuda}
	f(\nabla_iE^{\flat_j} - \nabla_jE^{\flat_i})=E^{\flat_i}\nabla_jf - E^{\flat_j}\nabla_if.
\end{eqnarray}

We can rewrite \eqref{construção1} using $\textnormal{curl}(fE)$. So,
\begin{eqnarray}\label{construção}
	fC_{ijk}&=&(R_{jl}\nabla^lfg_{ik}-R_{il}\nabla^lfg_{jk})+2(R_{ik}\nabla_jf-R_{jk}\nabla_if)+R(\nabla_ifg_{jk}-\nabla_jfg_{ik})\nonumber\\
	&&+\dfrac{f}{4}(\nabla_iRg_{jk}-\nabla_jRg_{ik})-2f(E^{\flat_j}\nabla_iE^{\flat_k}-E^{\flat_i}\nabla_jE^{\flat_k}).
\end{eqnarray}
Define the covariant $3$-tensor $V_{ijk}$ by
\begin{eqnarray}\label{tt}
	V_{ijk}&=&2f(E^{\flat_i}\nabla_jE^{\flat_k}-E^{\flat_j}\nabla_iE^{\flat_k})+\dfrac{f}{4}(\nabla_iRg_{jk}-\nabla_jRg_{ik})+R(\nabla_ifg_{jk}-\nabla_jfg_{ik})\nonumber\\
	&&-(R_{il}\nabla^lfg_{jk}-R_{jl}\nabla^lfg_{ik})-2(\nabla_ifR_{jk}-\nabla_jfR_{ik}),
\end{eqnarray}
where $E^{\flat_i}=E^\flat(e_i).$ The $V$-tensor has the same symmetries as the Cotton tensor $C$ and it is trace-free. Hence, from \eqref{construção} and \eqref{tt} we can conclude our next result.

\begin{lemma}
	Let $(M^3,\,g,\,f,\,E)$ be an electrostatic system. Then, 
	\begin{equation}\label{ttt3}
		fC_{ijk}=V_{ijk}.
	\end{equation}
\end{lemma}

Our next results follow the same strategy used by \cite{andrade}, \cite{catino2017gradient} and \cite{benedito}. We will sketch the proofs here for sake of completeness.
\begin{lemma}\label{lemma3.3}
	Let $(M^3,\,g,\,f,\,E)$ be an electrostatic system. Then, 
	\begin{equation*}
		C_{kji}R^{ik}=\nabla^i\nabla^k\left(\frac{V_{kij}}{f}\right).
	\end{equation*}
\end{lemma}
\begin{proof}
	In dimension $n=3$, the Bach tensor is defined by
	\begin{equation*}\label{bach3}
		B_{ij}=\nabla^kC_{kij}=\nabla^k\left(\frac{V_{kij}}{f}\right),
	\end{equation*}
	
	Taking the derivative over $i$, we have
	\begin{equation*}
		\nabla^iB_{ij}=\nabla^i\nabla^k\left(\frac{V_{kij}}{f}\right).
	\end{equation*}
	On other hand, 
	\begin{equation*}%\label{bc}
		\nabla^jB_{ij}=-C_{ijk}R^{jk}, \qquad C_{ijk}=-C_{jik},\qquad \nabla^kC_{kij}=\nabla^kC_{kji}
	\end{equation*}
	and 
	\begin{equation*}\label{soma}
		C_{ijk}+C_{kij}+C_{jki}=0.
	\end{equation*}
	Then, from a straightforward computation, we obtain
	\begin{equation*}
		\nabla^i\nabla^k\left(\frac{V_{kij}}{f}\right)=\nabla^iB_{ij}=-C_{jik}R^{ik}=-C_{jki}R^{ik}=C_{kji}R^{ik},
	\end{equation*}
	which is the expected result.
\end{proof}

\begin{lemma}\label{lemma4.3}
	Let $(M^3,\,g,\,f,\,E)$ be an electrostatic system. Then,
	\begin{equation*}
		\begin{aligned}
			\frac{1}{2}|C|^2+R^{ik}\nabla^jC_{jki}=-\nabla^j\nabla^i\nabla^k\left(\frac{V_{kij}}{f}\right).
		\end{aligned}
	\end{equation*}
\end{lemma}
\begin{proof}
	Taking the divergence in Lemma \ref{lemma3.3}, we get
	\begin{equation*}
		C_{kji}\nabla^jR^{ik}+R^{ik}\nabla^jC_{kji}=\nabla^j\nabla^i\nabla^k\left(\frac{V_{kij}}{f}\right).
	\end{equation*}
	Now, from the symmetries of the $C$-tensor and renaming indices, we can infer that
	\begin{equation*}%\label{rc}
		(\nabla^jR^{ik}-\nabla^kR^{ij})C_{jki}=C_{jki}\nabla^jR^{ik}+C_{kji}\nabla^kR^{ij}= 2C_{jki}\nabla^jR^{ik}.
	\end{equation*}
	Hence,
	\begin{equation*}
		\frac{1}{2}C_{kji}(\nabla^jR^{ik}-\nabla^kR^{ij})+R^{ik}\nabla^jC_{kji}=\nabla^j\nabla^i\nabla^k\left(\frac{V_{kij}}{f}\right).
	\end{equation*}
	Now, since the Cotton tensor is trace-free, from  \eqref{ct} we obtain
	\begin{equation*}
		-\frac{1}{2}C_{kji}C^{kji}-R^{ik}\nabla^jC_{jki}=\nabla^j\nabla^i\nabla^k\left(\frac{V_{kij}}{f}\right).
	\end{equation*}
	Therefore, the result holds.
\end{proof}

From this point, the structure of the electrostatic system plays an important role.
\begin{theorem}\label{teo3.3'}
	Let $(M^3,\,g,\,f,\,E)$ be an electrostatic system. For every $C^2$-function $\phi:\mathbb{R}\rightarrow\mathbb{R}$,  with $\phi(f)$ having compact support $K\subseteq M$ such that $K\cap\partial M=\emptyset$ we have
	\begin{eqnarray*}
		\frac{1}{4}\int_M\phi(f)|C|^2=\int_M\frac{\phi(f)}{f}\nabla^kf\nabla^i\nabla^jC_{jki}+\int_M\phi(f)E^{\flat_j}\nabla^kE^{\flat_i}C_{jki}.
	\end{eqnarray*}
\end{theorem}

\begin{proof}
	From Lemma \ref{lemma4.3}, we obtain
	\begin{equation*}
		\begin{aligned}
			\frac{1}{2}|C|^2\phi(f)+\phi(f)R^{ik}\nabla^jC_{jki}=-\phi(f)\nabla^j\nabla^i\nabla^k\left(\frac{V_{kij}}{f}\right).
		\end{aligned}
	\end{equation*}
	
	Integrating this expression, we get
	\begin{equation*}
		\begin{aligned}
			\frac{1}{2}\int_M|C|^2\phi(f)+\int_M\phi(f)R^{ik}\nabla^jC_{jki}=\int_M\dot{\phi}(f)\nabla^jf\nabla^i\nabla^k\left(\frac{V_{kij}}{f}\right).
		\end{aligned}
	\end{equation*}
	Thus, from Lemma \ref{lemma3.3} we have 
	\begin{equation*}
		\begin{aligned}
			\frac{1}{2}\int_M|C|^2\phi(f)+\int_M\phi(f)R^{ik}\nabla^jC_{jki}=&-\int_M\dot{\phi}(f)R^{ik}\nabla^jfC_{jki}.
		\end{aligned}
	\end{equation*}
	We will perform integration in some parts of the above equation, separately, using \eqref{combinado} and the fact that $C_{ijk}$ is trace-free and skew-symmetric. First,
	\begin{eqnarray*}
		\int_M\phi(f)R^{ik}\nabla^jC_{jki}&=&\int_M\dfrac{\phi(f)}{f}\nabla^i\nabla^kf\nabla^jC_{jki}-2\int_M\phi(f)E^{\flat_i}E^{\flat_k}\nabla^jC_{jki}\\
		&=&\int_M\dfrac{\phi(f)}{f}\nabla^i\nabla^kf\nabla^jC_{jki}+2\int_M\dot{\phi}(f)\nabla^jfE^{\flat_i}E^{\flat_k}C_{jki}\\
		&&+2\int_M\phi(f)\nabla^j(E^{\flat_i}E^{\flat_k})C_{jki}.
	\end{eqnarray*}
	On the other hand,
	\begin{eqnarray*}
		\int_M\dot{\phi}(f)R^{ik}\nabla^jfC_{jki}&=&\int_M\dfrac{\dot{\phi}(f)}{f}\nabla^jf\nabla^i\nabla^kfC_{jki}-2\int_M\dot{\phi}(f)\nabla^jfE^{\flat_i}E^{\flat_k}C_{jki}.\\
		%&=&\int_M\frac{\dot{\phi}(f)}{f}\nabla^k\nabla^jf\nabla^ifC_{jki}+\int_M\frac{\dot{\phi}(f)}{f}\nabla^jf\nabla^if\nabla^kC_{jki}\\ &&+2\int_M\dot{\phi}(f)\nabla^jfE^{\flat_i}E^{\flat_k}C_{jki}.
	\end{eqnarray*}
	Note that, since the Hessian tensor is symmetric
	\begin{equation*}%\label{hessiana}
		2\nabla^j\nabla^kf C_{jki}=\nabla^k\nabla^jf C_{jki}+\nabla^j\nabla^kf C_{kji}=\nabla^k\nabla^jf(C_{jki}+C_{kji})=0.
	\end{equation*}
	Hence,
	%\begin{equation*}
	\begin{align*}
		\frac{1}{2}\int_M|C|^2\phi(f)+&\int_M\dfrac{\phi(f)}{f}\nabla^i\nabla^kf\nabla^jC_{jki}+2\int_M\phi(f)\nabla^j(E^{\flat_i}E^{\flat_k})C_{jki}\\
		=&-\int_M\dfrac{\dot{\phi}(f)}{f}\nabla^jf\nabla^i\nabla^kfC_{jki}=\int_M\dfrac{\dot{\phi}(f)}{f}\nabla^jf\nabla^if\nabla^kC_{jki}\\
		=&-\int_M\dfrac{\dot{\phi}(f)}{f}\nabla^if\nabla^kf\nabla^jC_{jki}
		=-\int_M\dfrac{\nabla^i{\phi}(f)}{f}\nabla^kf\nabla^jC_{jki}\\
		=&-\int_M\frac{\phi(f)}{f^2}\nabla^if\nabla^kf\nabla^jC_{jki}+\int_M\dfrac{{\phi}(f)}{f}\nabla^i\nabla^kf\nabla^jC_{jki}\\
		&+\int_M\dfrac{{\phi}(f)}{f}\nabla^kf\nabla^i\nabla^jC_{jki}.
	\end{align*}
	Therefore, we get
	\begin{eqnarray}\label{integral}
		\int_M\frac{\phi(f)}{f}\nabla^kf\nabla^i\nabla^jC_{jki}&=&\frac{1}{2}\int_M|C|^2\phi(f)+\int_M\frac{\phi(f)}{f^2}\nabla^if\nabla^kf\nabla^jC_{jki}\nonumber\\
		&&+2\int_M\phi(f)\nabla^j(E^{\flat_i}E^{\flat_k})C_{jki}.
	\end{eqnarray}

	Then, since the Cotton tensor is trace-free and skew-symmetric, another integration by parts gives us
	
	\begin{eqnarray*}%\label{eqr01}
		\int_M\frac{\phi(f)}{f^2}\nabla^if\nabla^kf\nabla^jC_{jki}&=&\int_M\dfrac{\phi(f)}{f^2}\nabla^j\nabla^if\nabla^kfC_{kji}\\
		&=&\int_M\dfrac{\phi(f)}{f}(R^{ij}+2E^{\flat_i}E^{\flat_j})\nabla^kfC_{kji}.
	\end{eqnarray*}
	We used \eqref{combinado} in the last equality. Thus, \eqref{integral} can be rewrite in the following form:
	\begin{eqnarray*}
		\int_M\frac{\phi(f)}{f}\nabla^kf\nabla^i\nabla^jC_{jki}&=&\frac{1}{2}\int_M|C|^2\phi(f)+2\int_M\phi(f)\nabla^j(E^{\flat_i}E^{\flat_k})C_{jki}\nonumber\\
		&&-2\int_M\dfrac{\phi(f)}{f}E^{\flat_i}E^{\flat_j}\nabla^kfC_{jki}+\int_M\dfrac{\phi(f)}{f}R^{ij}\nabla^kfC_{kji}.
	\end{eqnarray*}
	
	Now, from \eqref{tt} and \eqref{ttt3}, we have
	\begin{eqnarray*}\label{ricci}
		R^{ij}\nabla^kfC_{kji}&=&\frac{1}{2}C_{kji}(\nabla^kfR^{ji}-\nabla^jf R^{ki})\nonumber\\
		&=&-\frac{1}{2}fC_{kji}\left[\dfrac{1}{2}C^{kji}+(E^{\flat_j}\nabla^kE^{\flat_i}-E^{\flat_k}\nabla^jE^{\flat_i})\right]\\
		&=&-\frac{1}{4}f|C|^2-\dfrac{1}{2}f\left(E^{\flat_j}\nabla^kE^{\flat_i}-E^{\flat_k}\nabla^jE^{\flat_i}\right)C_{kji}.\nonumber
	\end{eqnarray*}
	
	Thus,
	\begin{eqnarray*}
		&&\frac{1}{4}\int_M\phi(f)|C|^2=\int_M\frac{\phi(f)}{f}\nabla^kf\nabla^i\nabla^jC_{jki}\nonumber\\
		&&+2\int_M\dfrac{\phi(f)}{f}\left[E^{\flat_j}E^{\flat_i}\nabla^kf-f\nabla^j(E^{\flat_i}E^{\flat_k})+\dfrac{f}{4}\left(E^{\flat_k}\nabla^jE^{\flat_i}-E^{\flat_j}\nabla^kE^{\flat_i}\right)\right]C_{jki}.
	\end{eqnarray*}
	Furthermore, from \eqref{curlajuda} we have
	\begin{eqnarray*}
		E^{\flat_i}\nabla^kf - E^{\flat_k}\nabla^if =f(\nabla^iE^{\flat_k}-\nabla^kE^{\flat_i}).
	\end{eqnarray*}
	Combining the last two equations and the fact that Cotton tensor is skew-symmetric, yields to
	\begin{eqnarray*}
		&&\frac{1}{4}\int_M\phi(f)|C|^2=\int_M\frac{\phi(f)}{f}\nabla^kf\nabla^i\nabla^jC_{jki}\nonumber\\
		&&+2\int_M\phi(f)\left[E^{\flat_j}(\nabla^iE^{\flat_k}-\nabla^kE^{\flat_i})-\nabla^j(E^{\flat_i}E^{\flat_k})+\dfrac{1}{4}\left(E^{\flat_k}\nabla^jE^{\flat_i}-E^{\flat_j}\nabla^kE^{\flat_i}\right)\right]C_{jki}\\&&=\int_M\frac{\phi(f)}{f}\nabla^kf\nabla^i\nabla^jC_{jki}\nonumber\\
		&&+2\int_M\phi(f)\left[E^{\flat_j}\nabla^iE^{\flat_k}-E^{\flat_i}\nabla^jE^{\flat_k}-\frac{3}{4}E^{\flat_k}\nabla^jE^{\flat_i}-\frac{5}{4}E^{\flat_j}\nabla^kE^{\flat_i}\right]C_{jki}.
	\end{eqnarray*}
	%%%%%%%%%%%%%%%%%%%%%%%%%%%%%%%%%%%%%%%%
	Note that
	\begin{eqnarray*}
		&&E^{\flat_j}\nabla^iE^{\flat_k}C_{jki}  = - E^{\flat_k}\nabla^iE^{\flat_j}C_{jki},\qquad E^{\flat_i}\nabla^jE^{\flat_k}C_{jki} = - E^{\flat_i}\nabla^kE^{\flat_j}C_{jki},\\ &&E^{\flat_k}\nabla^jE^{\flat_i}C_{jki} = - E^{\flat_j}\nabla^kE^{\flat_i}C_{jki} \quad\mbox{and}\quad E^{\flat_j}\nabla^kE^{\flat_i}C_{jki} = - E^{\flat_k}\nabla^jE^{\flat_i}C_{jki}.
	\end{eqnarray*}
	Then,
	\begin{eqnarray*}
		\frac{1}{4}\int_M\phi(f)|C|^2&=&\int_M\phi(f)\left[2E^{\flat_j}\nabla^iE^{\flat_k}-2E^{\flat_i}\nabla^jE^{\flat_k}+E^{\flat_k}\nabla^jE^{\flat_i}\right]C_{jki}\nonumber\\
		&& + \int_M\frac{\phi(f)}{f}\nabla^kf\nabla^i\nabla^jC_{jki}.
	\end{eqnarray*}

	Since from \eqref{2forma} we have $$2E^{\flat_i}\nabla^jE^{\flat_k}C_{jki}=E^{\flat_i}E^{\flat_{jk}}C_{jki},$$ we can infer that
	\begin{eqnarray*}
		\frac{1}{4}\int_M\phi(f)|C|^2&=&\int_M\phi(f)\left[-2E^{\flat_k}\nabla^iE^{\flat_j}-E^{\flat_i}E^{\flat_{jk}}+E^{\flat_k}\nabla^jE^{\flat_i}\right]C_{jki}\\
		&&+\int_M\frac{\phi(f)}{f}\nabla^kf\nabla^i\nabla^jC_{jki}\\
		&=&\int_M\phi(f)\left[-E^{\flat_k}\nabla^iE^{\flat_j}+E^{\flat_i}E^{\flat_{kj}}+E^{\flat_k}E^{\flat_{ji}}\right]C_{jki}\\
		&&+\int_M\frac{\phi(f)}{f}\nabla^kf\nabla^i\nabla^jC_{jki}\\
		&=&\int_M\phi(f)\left[E^{\flat_k}E^{\flat_{ji}}+E^{\flat_j}\nabla^iE^{\flat_k}+E^{\flat_i}E^{\flat_{kj}}\right]C_{jki}\\
		&&+\int_M\frac{\phi(f)}{f}\nabla^kf\nabla^i\nabla^jC_{jki}\\
		&=&\int_M\phi(f)\left[E^{\flat_k}E^{\flat_{ji}}+E^{\flat_j}E^{\flat_{ik}}+E^{\flat_i}E^{\flat_{kj}}+E^{\flat_j}\nabla^kE^{\flat_i}\right]C_{jki}\\
		&&+\int_M\frac{\phi(f)}{f}\nabla^kf\nabla^i\nabla^jC_{jki}.
	\end{eqnarray*}
	
	On the other hand, since $E^\flat\wedge dE^\flat=0$, from \eqref{curlajuda} we get
	\begin{eqnarray*}
		f(E^{\flat_k}E^{\flat_{ji}}+E^{\flat_j}E^{\flat_{ik}}+E^{\flat_i}E^{\flat_{kj}})&=&E^{\flat_k}E^{\flat_j}\nabla^if-E^{\flat_k}E^{\flat_i}\nabla^jf+E^{\flat_j}E^{\flat_i}\nabla^kf\\&&-E^{\flat_j}E^{\flat_k}\nabla^if+E^{\flat_i}E^{\flat_k}\nabla^jf-E^{\flat_i}E^{\flat_j}\nabla^kf\\&=&0.
	\end{eqnarray*}
	
	Finally,
	\begin{eqnarray*}
		\frac{1}{4}\int_M\phi(f)|C|^2=\int_M\frac{\phi(f)}{f}\nabla^kf\nabla^i\nabla^jC_{jki}+\int_M\phi(f)E^{\flat_j}\nabla^kE^{\flat_i}C_{jki}.
	\end{eqnarray*}
\end{proof}

\section{Proof of the main results}\label{proofs}

It is well-known {(see \cite[Lemma 4]{tiarlos})} that in $\partial M=f^{-1}(0)$ the electric field and the gradient of the lapse function are linearly dependent (LD). Motivated by the {charged Nariai solution and the cold black hole system}, we assume that both fields are linearly dependent on $M$, that is, there exists a smooth function $\rho$ such that $E=\rho\nabla f$. Thus we can rewrite the $V$-tensor \eqref{tt} as follow.

\begin{lemma}\label{ld1}
	Let $(M^3,\,g,\,f,\,E)$ be an electrostatic system in which $E=\rho\nabla f$. Then the $V$-tenor is given by
	\begin{eqnarray*}
		V_{ijk}&=&f\rho|\nabla f|^2(\nabla_i\rho g_{jk}-\nabla_j\rho g_{ik})+\left(R-\dfrac{1}{2}Rf^2\rho^2-f^2\rho^2\Lambda\right)(\nabla_ifg_{jk}-\nabla_jfg_{ik})\\
		&&+2(f^2\rho^2-1)(\nabla_ifR_{jk}-\nabla_jfR_{ik})+(f^2\rho^2-1)(R_{il}\nabla^lfg_{jk}-R_{jl}\nabla^lfg_{ik}).
	\end{eqnarray*}
\end{lemma}
\begin{proof}
	Since the electric field and the lapse function are linearly dependent (LD), there exists a smooth function $\rho$ such that $E=\rho\nabla f$. Using \eqref{combinado} we get
	\begin{eqnarray*}
		\nabla_iE^{\flat_j}&=&\nabla_i(\rho \nabla_jf)\\
		&=&\nabla_i\rho\nabla_jf+\rho \nabla_i\nabla_jf\\
		&=&\nabla_i\rho\nabla_jf+2f\rho^3\nabla_jf\nabla_if+f\rho R_{ij}-\dfrac{f}{2}\rho Rg_{ij}.
	\end{eqnarray*}
	
	On other hand, from \eqref{rrr} we get
	\begin{eqnarray*}
		\nabla_i R&=&2\nabla_i|E|^2\\
		&=&4\rho|\nabla f|^2\nabla_i\rho+2\rho^2\nabla_i|\nabla f|^2.
	\end{eqnarray*}
	From \eqref{combinado} we know that
	$$\nabla_i|\nabla f|^2=2f\left(R_{il}\nabla^l f+2\rho^2|\nabla f|^2\nabla_i f-\frac{R}{2}\nabla_i f\right).$$
	Combining the last two equations and using \eqref{rrr}, we have
	\begin{eqnarray*}
		\nabla_i R&=&4\rho|\nabla f|^2\nabla_i\rho+4f\rho^2\left(R_{il}\nabla^l f+2\rho^2|\nabla f|^2\nabla_i f-\frac{R}{2}\nabla_i f\right).
	\end{eqnarray*}

	Then, from \eqref{tt} it follows that
	\begin{eqnarray*}
		V_{ijk}&=&2f\rho(\nabla_if\nabla_j\rho\nabla_kf-\nabla_jf\nabla_i\rho\nabla_kf)+f\rho|\nabla f|^2(\nabla_i\rho g_{jk}-\nabla_j\rho g_{ik})\\
		&&+2(f^2\rho^2-1)(\nabla_ifR_{jk}-\nabla_jfR_{ik})+(f^2\rho^2-1)(R_{il}\nabla^lfg_{jk}-R_{jl}\nabla^lfg_{ik})\\
		&&+\left[\left(1-\dfrac{3}{2}f^2\rho^2\right)R+2f^2\rho^4|\nabla f|^2\right](\nabla_ifg_{jk}-\nabla_jfg_{ik}).
	\end{eqnarray*}
	Since $\textnormal{curl}(fE)=0$ implies that $E^{\flat_{ij}}=0$ (the fields are LD, i.e., $df\wedge E^\flat=0$), we have $\nabla_iE^{\flat_j}=\nabla_jE^{\flat_i}.$ So, $$\nabla_i\rho\nabla_jf=\nabla_j\rho\nabla_if.$$
	
	Finally, combining \eqref{rrr} with the last two identities the result follows.
\end{proof}

Moreover, from Theorem \ref{teo3.3'} we obtain the following corollary.
\begin{corollary}\label{corolario2}
	Let $(M^3,\,g,\,f,\,E)$ be an electrostatic system where the electric field and gradient of the lapse function are linearly dependent. For every $\phi:\mathbb{R}\rightarrow\mathbb{R}$, $C^2$
	function with $\phi(f)$ having compact support $K\subseteq M$ we have
	
	\begin{eqnarray*}
		\frac{1}{2}\int_M\dfrac{1}{Q}|C|^2\phi(f)=\int_M\frac{\phi(f)}{f}\nabla^kf\nabla^i\nabla^jC_{jki},
	\end{eqnarray*}
	where $Q=2(1-f^2\rho^2)\neq0$.
\end{corollary}
\begin{proof}
	Taking in accounting that $E=\rho\nabla f$ in Theorem \ref{teo3.3'}, since the Cotton tensor is skew-symmetric and trace-free we obtain 
	\begin{eqnarray*}
		\frac{1}{4}\int_M\phi(f)|C|^2&=&\int_M\frac{\phi(f)}{f}\nabla^kf\nabla^i\nabla^jC_{jki}+\int_M\phi(f)f\rho^2\nabla^jfR^{ki}C_{jki},
	\end{eqnarray*}
	where we used that $\textnormal{curl}(fE)=0$, i.e., $$\nabla^k\rho\nabla^if=\nabla^i\rho\nabla^kf.$$

	On other hand, from Lemma \ref{ld1} we have
	
	\begin{eqnarray*}
		\nabla^jfR^{ki}C_{jki}&=&\dfrac{1}{2}C_{jki}(\nabla^jfR^{ki}-\nabla^kfR^{ji})\\
		&=&\dfrac{1}{2Q}C_{jki}V^{jki}\\
		&=&\dfrac{1}{2Q}f|C|^2.
	\end{eqnarray*}
\end{proof}

Now, we are able to demonstrate our main results.

\begin{theorem}[Theorem \ref{teoharmonico'}]
		Let $(M^3,\,g,\,f,\,E)$ be a compact electrostatic system such that the electric field and the gradient of the lapse function are linearly dependent. Suppose that the Bach tensor is divergence-free and $Q$ has defined sign everywhere. Then $(M^3,\,g)$ is locally conformally flat. 
\end{theorem}

\begin{proof}
	Considering $M$ is compact {without boundary} and $\phi(f)=f^4$, from Corollary \ref{corolario2} we obtain
	\begin{eqnarray*}
		\frac{1}{2}\int_M\dfrac{1}{Q}|C|^2f^4&=&\int_Mf^3\nabla^kf\nabla^i\nabla^jC_{jki}\\
		&=&\frac{1}{4}\int_M\nabla^if^4\nabla^k\nabla^jC_{jki}\\
		&=&-\frac{1}{4}\int_Mf^4\nabla^i\nabla^k\nabla^jC_{jki}.
	\end{eqnarray*}
	Since  $\textnormal{div}^2B = 0$, by definition,  $\textnormal{div}^3C = 0$, then the right-hand side is identically zero, i.e., 
	\begin{eqnarray*}
		\int_M\dfrac{1}{Q}|C|^2f^4=0,
	\end{eqnarray*}
	
	Since $f>0$ over $M$ and $Q$ has defined sign everywhere, $f^2\rho^2\neq1$, that is, the integral has always the same sign, therefore  $C$ must be identically zero, thus the result holds.
	
\end{proof}

\begin{theorem}[Theorem \ref{proper'}]
	Let $(M^3,\,g,\,f,\,E)$ be an electrostatic system such that the electric field and the gradient of the lapse function are linearly dependent. Suppose that the Bach tensor is divergence-free and $Q$ has defined sign everywhere. If $f$ is a proper function, then  $(M^3,\,g)$ is locally conformally flat.
\end{theorem}
\begin{proof}
	Let $s>0$ be a real number fixed, and so we take $\chi\in C^3$ a real non-negative function defined by $\chi=1$ in {$[0,s]$}, $\chi'\leq0$ in $[s,2s]$ and $\chi=0$ in $[2s,+\infty]$. Since $f$ is a proper function, we have that $\phi(f)=f^4\chi(f)$ has compact support in $M$ for $s>0$. From Corollary \ref{corolario2}, we get
	\begin{eqnarray*}
		-\frac{1}{2}\int_M\dfrac{1}{Q}|C|^2f^4\chi(f)&=&\int_Mf^3\chi(f)\nabla^kf\nabla^i\nabla^jC_{jki}\\&=&\frac{1}{4}\int_M\chi(f)\nabla^if^4\nabla^k\nabla^jC_{jki}\\
		&=&-\frac{1}{4}\int_M\chi(f)f^4\nabla^i\nabla^k\nabla^jC_{jki}\\
		&+&\frac{1}{4}\int_M\dot{\chi}(f)f^4\nabla^if\nabla^k\nabla^jC_{jki}.
	\end{eqnarray*}
	
	In the last equality we used integration by parts. Now, since $\textnormal{div}^2 B=0$ and taking $\phi(f)=f^5\dot{\chi}(f)$ in Corollary \ref{corolario2} one more time, we obtain 
	\begin{eqnarray*}
		\frac{1}{2}\int_M\dfrac{1}{Q}|C|^2f^4\chi(f)&=&-\frac{1}{8}\int_M\dfrac{1}{Q}|C|^2f^5\dot{\chi}(f),
	\end{eqnarray*}
	i.e.,
	\begin{equation*}
		\begin{aligned}
			\int_M\dfrac{1}{Q}f^4|C|^2[\chi(f)+\frac{1}{4}f\dot{\chi}(f)]=0.
		\end{aligned}
	\end{equation*}
	
	Let be $M_s=\{x\in M; f(x)\leq s\}$. Thus, by definition, $\chi(f)+\frac{1}{4}f\dot{\chi}(f)=1$ on the compact set $M_s$. Thus, on $M_s$,
	$$\int_{M_s}\dfrac{1}{Q}f^4|C|^2=0.$$
	Therefore, since $Q$ has defined sign everywhere and $f$ is positive, $C=0$ in $M_s$. Taking $s\rightarrow+\infty$, we obtain that $C=0$ on $M$.
\end{proof}

\section{The Warped Product Structure}\label{warped}

In this section we will prove the warped product structure of a locally conformally flat electrostatic system following the ideas of \cite{cao, sun}. Consequently, the proof of Theorem \ref{fiber007-3} is given.

We consider an orthonormal frame $\{e_{1}, e_{2}, e_{3}\}$ diagonalizing the Ricci tensor $\textnormal{Ric}$ at a regular point $p\in\Sigma=f^{-1}(c)$, with associated eigenvalues $R_{kk}$, $k=1,\,2,\, 3,$ respectively. That is, $R_{ij}(p)=R_{ii}\delta_{ij}(p)$.
Now, from Theorem \ref{teoharmonico'} and Theorem \ref{proper'} we can infer that $V_{ijk}=0$ (since $(M,\,g)$ is locally conformally flat). Then, from Lemma \ref{ld1}, for all $i\neq j$ we get
\begin{eqnarray}\label{Vautovalor}
	0=V_{ijj}&=&f\rho|\nabla f|^2\nabla_i\rho   +(f^2\rho^2-1)(2R_{jj} + R_{ii})\nabla_if\nonumber\\
	&&+\left(R-\dfrac{1}{2}Rf^2\rho^2-f^2\rho^2\Lambda\right)\nabla_if.
\end{eqnarray}

Without loss of generalization, consider $\nabla_{i}f\neq0$ and $\nabla_{j}f=0$ for all $i\neq j$. Observe that $\textnormal{Ric}(\nabla f)=R_{ii}\nabla f$, i.e., $\nabla f$ is an eigenvector for $\textnormal{Ric}$. From \eqref{Vautovalor}, we obtain that $R_{ii}$ and $R_{jj}, \ j\neq i,$ have multiplicity $1$ and $2$, respectively. In fact,
\begin{eqnarray*}
	-f\rho|\nabla f|^2\frac{\nabla_1\rho}{\nabla_1f}  -\left(R-\dfrac{1}{2}Rf^2\rho^2-f^2\rho^2\Lambda\right)-(f^2\rho^2-1)R_{11}=2(f^2\rho^2-1)R_{jj},
\end{eqnarray*}
for $j=2,\,3.$

Moreover, suppose that $\nabla_{i}f\neq0$ for at least two distinct directions. Assume $\nabla_1f\neq0$, $\nabla_2f\neq0$ and $\nabla_3f=0$. So, for instance, we have
\begin{eqnarray*}
	-f\rho|\nabla f|^2\frac{\nabla_1\rho}{\nabla_1f}  -\left(R-\dfrac{1}{2}Rf^2\rho^2-f^2\rho^2\Lambda\right)-(f^2\rho^2-1)R_{11}=2(f^2\rho^2-1)R_{33}
\end{eqnarray*}
and
\begin{eqnarray*}
	-f\rho|\nabla f|^2\frac{\nabla_2\rho}{\nabla_2f}  -\left(R-\dfrac{1}{2}Rf^2\rho^2-f^2\rho^2\Lambda\right)-(f^2\rho^2-1)R_{22}=2(f^2\rho^2-1)R_{33}.
\end{eqnarray*}
Then, using that $\textnormal{curl}(fE)=0$, i.e., $$\nabla^k\rho\nabla^if=\nabla^i\rho\nabla^kf.$$ We can conclude that $$\frac{\nabla_1\rho}{\nabla_1f}=\frac{\nabla_2\rho}{\nabla_2f}.$$ Thus, $R_{11}=R_{22}.$ Analogously, if $\nabla_if\neq0$ for all $i\in\{1,\,2,\,3\}$. Then, $R_{11}=R_{22}=R_{33}.$ So, we can conclude that $Ric$ has at most two distinct eigenvalues $\lambda$ and $\mu$ with one of them having multiplicity $2$. 

Therefore, in any case we have that $\nabla f$ is an eigenvector for $\textnormal{Ric}$. From the above discussion we can take $\{e_{1}=\frac{\nabla f}{|\nabla f|},e_{2},\,e_{3}\}$ as an orthonormal frame for $\Sigma$ diagonalizing the Ricci tensor $\textnormal{Ric}$ for the metric $g$. 

Now, from \eqref{combinado} we have
\begin{eqnarray*}
	\nabla_a|\nabla f|^2=2f\left(R_{al}\nabla^l f+2\rho^2|\nabla f|^2\nabla_a f-\frac{R}{2}\nabla_a f\right);\quad a\in\{2,\,3\}.
\end{eqnarray*}
Hence, $|\nabla f|$ is a constant in $\Sigma$. Thus, we can express locally the metric $g$ in the form
\begin{eqnarray*}
	g_{ij} = \frac{1}{|\nabla f|^{2}}df^{2} + g_{ab}(f,\theta)d\theta_{a}d\theta_{b},
\end{eqnarray*}
where $g_{ab}(f, \theta)d\theta_{a}d\theta_{b}$ is the induced metric and $(\theta_{2},\,\theta_{3})$ is any local coordinate system on $\Sigma$. We can find a good overview of the level set structure in \cite{cao,benedito}.

Observe that there is no open subset $\Omega$ of $M^{n}$ where $\{\nabla f=0\}$ is dense. In fact, if $f$ is constant in $\Omega$ and $M^{n}$ is complete, we have that $f$ is analytic, which implies $f$ is constant everywhere. Thus, we consider $\Sigma$ a connected component of the level surface $f^{-1}(c)$ (possibly disconnected) where $c$ is any regular value of the function $f$. Suppose that $I$ is an open interval containing $c$ such that $f$ has no critical points in the open neighborhood $U_{I}=f^{-1}(I)$ of $\Sigma$. For sake of simplicity, let $U_{I}\subset M\backslash\{f=0\}$ be a connected component of $f^{-1}(I)$. Then, we can make a change of variables 
\begin{eqnarray*}
	r(x)=\int\frac{df}{|\nabla f|}
\end{eqnarray*}
such that the metric $g$ in $U_{I}$ can be expressed by
\begin{eqnarray*}
	g_{ij}=dr^{2}+g_{ab}(r,\theta)d\theta_{a}d\theta_{b}.
\end{eqnarray*}

Let $\nabla r=\frac{\partial}{\partial r}$, then $|\nabla r|=1$ and $\nabla f=f'(r)\frac{\partial}{\partial r}$ on $U_{I}$. Note that $f^{\prime}(r)$ does not change
sign on $U_{I}$. Thus, we may assume $I = (-\varepsilon,\,\varepsilon)$
with $f'(r)>0$ for $r\in I$. Moreover, we have $\nabla_{\partial r}\partial r=0.$

Then the second fundamental formula on $\Sigma$ is given by
\begin{eqnarray}\label{eq555}
	h_{ab}&=& - \langle e_{1},\,\nabla_{a}e_{b}\rangle=\frac{\nabla_{a}\nabla_{b}f}{|\nabla f|}\nonumber\\
	&=&\frac{1}{|\nabla f|}\left(f R_{ab}-\frac{Rf}{2}g_{ab}\right)=\frac{f}{|\nabla f|}\left(\mu-\frac{R}{2}\right)g_{ab}=\frac{H}{2}g_{ab},
\end{eqnarray}
  where $H=H(r)$, since $H$ is constant in $\Sigma$. In fact, contracting the Codazzi equation 
\begin{eqnarray*}
	R_{1cab}=\nabla_{a}h_{bc}-\nabla_{b}h_{ac}
\end{eqnarray*}
over $c$ and $b$, it gives
\begin{eqnarray*}
	R_{1a}=\nabla_{a}(H)-\frac{1}{2}\nabla_{a}(H)=\frac{1}{2}\nabla_{a}(H).
\end{eqnarray*}
On the other hand, since $R_{1a}=0,$ we conclude that $H$ is constant in $\Sigma$.

For what follows, we fix a local coordinate system
$$(x_{1},\, x_2,\, x_{3}) = (r,\,\theta_2,\, \theta_{3}) $$
in $U_{I}$, where $(\theta_{2},\theta_{3})$ is any local coordinate system on the level surface $\Sigma_{c}$. Considering that $a, b, c,\cdots\in\{2 , 3\}$, we have
\begin{eqnarray*}
	h_{ab}=-g(e_1,\, \nabla_{a}\partial_{b})=-g(e_1, \Gamma^{1}_{ab}\partial_{r})=\frac{-1}{|\nabla f|}\Gamma^{1}_{ab}.
\end{eqnarray*}
Now, by definition
\begin{eqnarray*}
	\Gamma^{1}_{ab}=\frac{1}{2}g^{11}\left(-\frac{\partial}{\partial r}g_{ab}\right)=\frac{1}{2}|\nabla f|\frac{\partial}{\partial r}g_{ab}.
\end{eqnarray*}
Then,
\begin{eqnarray*}
	\frac{\partial}{\partial r}g_{ab}= - H(r)g_{ab}
\end{eqnarray*}
implies that
\begin{eqnarray*}
	g_{ab}(r,\theta)=\varphi(r)^{2}g_{ab}(r_{0},\theta),
\end{eqnarray*}
where $\varphi(r)=e^{\left(-\int^{r}_{r_{0}}H(s)ds\right)}$ and the level set $\{r=r_{0}\}$  corresponds to the connected component $\Sigma$ of $f^{-1}(c)$.

Now, we can apply the warped product structure (see \cite{besse}). Hence, considering  $$(M^{3},\,g)=(I,\,dr^{2})\times_{\varphi}(N^{2},\,\overline{g}),$$ where $g=dr^2+\varphi^2\overline{g}.$ 
The Ricci tensor of $(M^3,\,g)$ is
\begin{eqnarray}\label{wps}
	R_{11}=-2\frac{\varphi''}{\varphi},\qquad R_{1a}=0
\end{eqnarray}
and
\begin{eqnarray*}
	R_{ab}=\overline{R}_{ab}-\left[(\varphi')^2+\varphi\varphi''\right]\bar{g}_{ab}\qquad  (a,\,b\in\{2,\,3\}).
\end{eqnarray*}
Since $\overline{R}_{ab}=\frac{\overline{R}}{2}\overline{g}_{ab}$,
\begin{eqnarray*}
	R_{ab}=\left[\frac{\overline{R}}{2}-(\varphi')^2-\varphi\varphi''\right]\overline{g}_{ab}.
\end{eqnarray*}
On the other hand, {since 
	\begin{eqnarray*}
		R=\varphi^{-2}\overline{R}-2\left(\frac{\varphi'}{\varphi}\right)^2-4\frac{\varphi''}{\varphi}, 
	\end{eqnarray*}
	we get}
 \begin{eqnarray*}\label{Rbar=R}
	\overline{R} = \varphi^2R+2(\varphi')^2+4\varphi\varphi''.
\end{eqnarray*}

\iffalse
\begin{eqnarray*}\label{Rbar=R}
	\overline{R} = \varphi^2R+2\left(\frac{\varphi'}{\varphi}\right)^2+4\frac{\varphi''}{\varphi}.
\end{eqnarray*}
\fi

Since $R=2(\rho^2|\nabla f|^2+\Lambda)$ we get
\begin{eqnarray}\label{barR}
	\overline{R} = 2\varphi^2\rho^2(f')^2 + 2(\varphi')^2 + 4\varphi\varphi'' + 2\varphi^2\Lambda.
\end{eqnarray}

Moreover, from \eqref{combinado} we know that
$$\frac{1}{|\nabla f|^2}\langle\nabla|\nabla f|^2,\,\nabla f\rangle=2f\left(R_{11}+2\rho^2|\nabla f|^2-\frac{R}{2}\right).$$ 
That is, from \eqref{rrr} and \eqref{wps} we get
\begin{eqnarray*}
	\langle\nabla|\nabla f|^2,\,\nabla f\rangle=2f(f')^2\left[\rho^2(f')^2-2\frac{\varphi''}{\varphi} - \Lambda\right].
\end{eqnarray*}
Hence, using that $\nabla f = f'\partial_r$ we obtain
\begin{eqnarray*}
	2(f')^2f''=2f(f')^2\left[\rho^2(f')^2-2\frac{\varphi''}{\varphi} - \Lambda\right].
\end{eqnarray*}
So,
\begin{eqnarray*}
	\rho^2=\frac{1}{(f')^2}\left[\frac{f''}{f}+2\frac{\varphi''}{\varphi}+\Lambda\right].
\end{eqnarray*}

Combining the above identity with \eqref{barR} we can conclude that $\overline{R}$ does not depend on $\theta$. Therefore, $\overline{R}$ is a constant.

%%%%%%%%%%%%%%%%%%%%%%%%%%%%%%%%%%%%%%%%%%%%%%%%%%%%%%%%%%%%%%%%%%%%%%%%%%%%%%%%%%%%%%%%%%%%%%%%%%%%%%%%%%%%%%%%%%%%%%%%%%%%%%%%%%%%%%%%%%%%%%%%%%%%%%%%%%%%%%%%%%
\iffalse
Moreover, we know that $\textnormal{div}(E)=0$ implies that $\langle\nabla\rho,\,\nabla f\rangle+\rho\Delta f=0$. So, {from \eqref{s1},} $f'\rho'+\rho(\rho^2(f')^2-\Lambda)f=0.$ Thus, we have
\begin{eqnarray*}
	\rho\left(f''+2\frac{\varphi''}{\varphi}f\right)+\rho'f'=0.
\end{eqnarray*}  
\fi

%%%%%%%%%%%%%%%%%%%%%%%%%%%%%%%%%%%%%%%%%%%%%%%%%%%%%%%%%%%%%%%%%%%%%%%%%%%%%%%%%%%%%%%%%%%%%%%%%

\iffalse
A straightforward computation proves (see \cite[Equation 2.1]{dominguez2018introduction}) that the mean curvature of $\Sigma$ also is given by
\begin{eqnarray*}
	H = |\nabla f|^{-1}\left(\Delta f -  \frac{1}{2|\nabla f|^2}\langle\nabla|\nabla f|^2,\,\nabla f\rangle\right).
\end{eqnarray*}
Therefore, 
\begin{eqnarray}\label{H1}
	H= \frac{1}{f'}\left[(\rho^2(f')^2-\Lambda)f- f''\right].
\end{eqnarray}
\fi

%%%%%%%%%%%%%%%%%%%%%%%%%%%%%%%%%%%%%%%%%%%%%%%%%%%%%%%%%%%%%%%%%%%%%%%%%%%%%%%%%%%%%%%%%%%%%%%%%%%%%%%%%%%%%%%%%%%%%%%%%%%%%%%%%%%%%%%%%%%%%%%%%%%%%%%%%%%%%%%%%%%%%%%%%%%%%%%%%%%%%%%%%%%%%%%%%%%%%%

Furthermore, from \eqref{eq555} we have
\begin{eqnarray*}
	\frac{1}{2}|\nabla f|Hg_{ab}=\nabla_a\nabla_bf=f\left(R_{ab}+2\rho^2\nabla_af\nabla_bf-\frac{R}{2}g_{ab}\right).
\end{eqnarray*}
Thus,
\begin{eqnarray*}
	\left(\frac{1}{2}|\nabla f|H+\frac{Rf}{2}\right)\varphi^2\overline{g}_{ab}=fR_{ab}.
\end{eqnarray*}
On the other hand, \begin{eqnarray*}
	fR_{ab}=f\left[\frac{\overline{R}}{2}-(\varphi')^2-\varphi\varphi''\right]\overline{g}_{ab}.
\end{eqnarray*}
Then,
\begin{eqnarray*}
	f\left[\frac{1}{2}\varphi^2R + \varphi\varphi''\right]=\left(\frac{1}{2}f'H+\frac{Rf}{2}\right)\varphi^
	2,
\end{eqnarray*}
i.e.,
\begin{eqnarray*}
	H=2\frac{f}{f'}\frac{\varphi''}{\varphi}.
\end{eqnarray*}

Since $\varphi=e^{-\int^r_{r_0} H(s)ds}$, we conclude
\begin{eqnarray*}
	\varphi' + 2\frac{f}{f'}\varphi''=0\quad\Rightarrow\quad\varphi'(r)=c_1{f(r)}^{-1/2},
\end{eqnarray*}
where $c_1\in\mathbb{R}.$

\bibliographystyle{amsplain}
%\bibliography{ijmsample}

%%%%%%%%%%%AAAAAAAAAAAAAAAAAAAAAAAAAA%%%%%%%%%%%%%%%%%%%%%%%%%%%%%%%%%%%%%%%%%%%%
@article{andrade,
  title={{On the geometry of electrovacuum spaces in higher dimensions}},
  author={Andrade, Maria and Leandro, Benedito and Lousa, R{\'o}bson},
  journal={arXiv preprint arXiv:2107.14605},
  pages={1--23},
  year={2021}
}
@article{cao2012locally,
  title={{On locally conformally flat gradient steady Ricci solitons}},
  author={Cao, Huai-Dong and Chen, Qiang},
  journal={Transactions of the American Mathematical Society},
  volume={364},
  number={5},
  pages={2377--2391},
  year={2012}
}

@article{anderson2000structure,
  title={{On the structure of solutions to the static vacuum Einstein equations}},
  author={Anderson, Michael T},
  journal={arXiv preprint gr-qc/0001018},
  year={2000}
}

@article{ambrozio2017static,
  title={{On static three-manifolds with positive scalar curvature}},
  author={Ambrozio, Lucas},
  journal={Jour. Diff. Geom.},
  volume={107},
  number={1},
  pages={1--45, MR3698233},
  year={2017},
  publisher={Lehigh University}
}










%%%%%%%%%%%%%%%BBBBBBBBBBBBBBBBBBBB%%%%%%%%%%%%%%%%%%%%%%%%%%%%%%%%%%%%%%%%%%%%
@article{ernani,
  title={{Isoperimetric inequality and Weitzenb{\"o}ck type formula for critical metrics of the volume}},
  author={Baltazar, Halyson and Di{\'o}genes, Rafael and Ribeiro, Ernani},
  journal={Israel Journal of Mathematics},
  volume={234},
  number={1},
  pages={309--329, MR4040829},
  year={2019},
  publisher={Springer}
}

@book{besse,
  title={{Einstein manifolds}},
  author={Besse, Arthur L},
  year={2007, MR2371700},
  publisher={Springer Science \& Business Media}
}

@article{beyer,
  title={Non-genericity of the Nariai solutions: I. Asymptotics and spatially homogeneous perturbations},
  author={Beyer, Florian},
  journal={Classical and Quantum Gravity},
  volume={26},
  number={23},
  pages={235015},
  year={2009},
  publisher={IOP Publishing}
}

@article{brozos,
  title={{Half conformally flat generalized quasi-Einstein manifolds of metric signature (2, 2)}},
  author={Brozos-V{\'a}zquez, Miguel and Garc{\'\i}a-R{\'\i}o, Eduardo and Gilkey, Peter and Valle-Regueiro, Xabier},
  journal={International Journal of Mathematics},
  volume={29},
  number={01},
  pages={MR3756414},
  year={2018},
  publisher={World Scientific}
}

@article{brendle,
  title={{Uniqueness of gradient Ricci solitons}},
  author={Brendle, Simon},
  journal={arXiv preprint arXiv:1010.3684},
  pages={ MR2802586},
  year={2010}
}

@article{brozos2005some,
  title={Some remarks on locally conformally flat static space--times},
  author={Brozos-V{\'a}zquez, Miguel and Garc{\'\i}a-R{\'\i}o, Eduardo and V{\'a}zquez-Lorenzo, Ram{\'o}n},
  journal={Journal of mathematical physics},
  volume={46},
  number={2},
  pages={022501},
  year={2005},
  publisher={American Institute of Physics}
}



%\bibitem{burko2021}{Burko, L. M.; Khanna, G.; Sabharwal, S.} {\em Scalar and gravitational hair for extreme Kerr black holes.} Phys. Rev. D 103 (2021), no. 2, L021502, 5 pp. MR4213220.
%%%%%%%%%%%%%%%%%%%%CCCCCCCCCCCCCCCCCCCCCCCCCCCCCCCCCCCCCCCCCc%%%%%%%%%%%%%%%%%%%%%%%%%%%%%%%%%5

@article{cao,
  title={{On Bach-flat gradient shrinking Ricci solitons}},
  author={Cao, Huai-Dong and Chen, Qiang},
  journal={Duke Math. J.},
  volume={162},
  number={6},
  pages={1149--1169, MR3053567},
  year={2013},
  publisher={Duke University Press}
}

@article{cardoso2004nariai,
  title={{Nariai, Bertotti-Robinson, and anti-Nariai solutions in higher dimensions}},
  author={Cardoso, Vitor and Dias, Oscar JC and Lemos, Jos{\'e} PS},
  journal={Physical review D},
  volume={70},
  number={2},
  pages={024002},
  year={2004},
  publisher={APS}
}

@article{cederbaum2016uniqueness,
  title={Uniqueness of photon spheres in electro-vacuum spacetimes},
  author={Cederbaum, Carla and Galloway, Gregory J},
  journal={Classical and Quantum Gravity},
  volume={33},
  number={7},
  pages={075006},
  year={2016},
  publisher={IOP Publishing}
}

@article{cardoso2009instability,
  title={Instability of Reissner-Nordstr{\"o}m black holes in de Sitter backgrounds},
  author={Cardoso, Vitor and Lemos, Madalena and Marques, Miguel},
  journal={Physical Review D},
  volume={80},
  number={12},
  pages={127502},
  year={2009},
  publisher={APS}
}


@article{sun,
  title={{On the structure of gradient Yamabe solitons}},
  author={Cao, Huai-Dong and Sun, Xiaofeng and Zhang, Yingying},
  journal={Math. Res. Lett.},
   volume={19},
  number={4},
  pages={767-774, MR3008413.},
  year={2012}
}

@article{catino,
  title={{Generalized quasi-Einstein manifolds with harmonic Weyl tensor}},
  author={Catino, Giovanni},
  journal={Math. Z.},
  volume={271},
  number={3},
  pages={751--756. MR2945582},
  year={2012},
  publisher={Springer}
}

@article{catino2016,
  title={{Integral pinched shrinking Ricci solitons}},
  author={Catino, Giovanni},
  journal={Advances in Mathematics},
  volume={303},
  pages={279--294, MR3552526},
  year={2016},
  publisher={Elsevier}
}

@article{catino2012,
  title={{A note on four-dimensional (anti-) self-dual quasi-Einstein manifolds}},
  author={Catino, Giovanni},
  journal={Differential Geometry and its Applications},
  volume={30},
  number={6},
  pages={660--664, MR2996860},
  year={2012},
  publisher={Elsevier}
}

@article{catinoman,
  title={{The evolution of the Weyl tensor under the Ricci flow}},
  author={Catino, Giovanni and Mantegazza, Carlo},
  journal={Ann. Inst. Fourier,},
  volume={61},
  number={4},
  pages={1407--1435, MR2951497},
  year={2011}
}

@article{chen2015four,
  title={{On four-dimensional anti-self-dual gradient Ricci solitons}},
  author={Chen, Xiuxiong and Wang, Yuanqi},
  journal={The Journal of Geometric Analysis},
  volume={25},
  number={2},
  pages={1335--1343, MR3319974},
  year={2015},
  publisher={Springer}
}

@article{catinomant,
  title={Locally conformally flat quasi-Einstein manifolds},
  author={Catino, Giovanni and Mantegazza, Carlo and Mazzieri, Lorenzo and Rimoldi, Michele},
  journal={Journal f{\"u}r die reine und angewandte Mathematik (Crelles Journal)},
  volume={2013},
  number={675},
  pages={181--189},
  year={2013},
  publisher={De Gruyter}
}

@article{tiarlos,
  title={Min-max minimal surfaces, horizons and electrostatic systems},
  author={Cruz, Tiarlos and Lima, Vanderson and de Sousa, Alexandre},
  journal={to appear in Jour. Diff. Geom.,},
  pages={1--47},
  year={arXiv:1912.08600}
}

@book{chow2006hamilton,
  title={{Hamilton's Ricci flow}},
  author={Chow, Bennett and Lu, Peng and Ni, Lei},
  volume={77},
  year={2006, MR2274812},
  publisher={American Mathematical Soc.}
}

@article{catino2017gradient,
  title={{Gradient Ricci solitons with vanishing conditions on Weyl}},
  author={Catino, Giovanni and Mastrolia, Paolo and Monticelli, Dario Daniele},
  journal={Journal De Math{\'e}matiques Pures Et Appliqu{\'e}es},
  volume={108},
  number={1},
  pages={1--13},
  year={2017},
  publisher={Elsevier}
}

%\bibitem{catino2} Catino, G.; Mastrolia, P.; Monticellia, D. D. \emph{Gradient Ricci solitons with vanishing conditions on Weyl}. J. Math. Pures Appl. 108 (2017): 1-13. MR3660766.


%\bibitem{cederbaum2016}{Cederbaum, C.; Galloway, G. J.} \emph{Uniqueness of photon spheres in electro-vacuum space-times.} Classical Quantum Gravity 33 (2016), no. 7, 075006, 16 pp. MR3471730.

%\bibitem{cederbaum2017}{Cederbaum, C.; Galloway, G. J.} {\em Uniqueness of photon spheres via positive mass rigidity.} Comm. Anal. Geom. 25(2), (2017): 303-320. MR3690243
	
%\bibitem{chrusciel2007}{Chru\'sciel, P. T.; Tod, P.} \emph{The classification of static electro-vacuum space-times containing an asymptotically flat spacelike hypersurface with compact interior.} Comm. Math. Phys. 271 (2007), no. 3, 577-589. MR2291788.

@article{chrusciel1999,
  title={Towards a classification of static electrovacuum spacetimes containing an asymptotically flat spacelike hypersurface with compact interior},
  author={Chrusciel, Piotr T.},
  journal={Classical and Quantum Gravity},
  volume={16},
  number={3},
  pages={689--704. MR1682570},
  year={1999},
  publisher={IOP Publishing}
}

@article{chrusciel2017non,
  title={{Non-singular space-times with a negative cosmological constant: II. Static solutions of the Einstein--Maxwell equations}},
  author={Chru{\'s}ciel, Piotr T. and Delay, Erwann},
  journal={Lett. Math. Phys.},
  volume={107},
  number={8},
  pages={1391--1407, MR3669238},
  year={2017},
  publisher={Springer}
}

@article{chrusciel2005non,
  title={{Non-singular, vacuum, stationary space-times with a negative cosmological constant}},
  author={Chrusciel, Piotr T. and Delay, Erwann},
  journal={Annales Henri Poincare},
   volume={8},
  number={},
  pages={219--239, MR2314449},
  year={2007}
}



%%%%%%%%%%%%%%%%%%%%%%%DDDDDDDDDDDDDDDDDDDDDDDDDDDDDDDDDDDDDDDDDDDDDDd%%%%%%%%%%%%%%%%%%%%%%%%%%%%%%%%%%%%%%%%%%%%%%%%

@article{dias2020,
  title={{Origin of the Reissner-Nordstr{\"o}m--de Sitter instability}},
  author={Dias, Oscar JC and Santos, Jorge E},
  journal={Physical Review D},
  volume={102},
  number={12},
  pages={124039, MR4197006},
  year={2020},
  publisher={APS}
}


@inproceedings{guha2012geodesic,
  title={Geodesic motions near a five-dimensional reissner--nordstr{\"o}m anti-de sitter black hole},
  author={Guha, Sarbari and Bhattacharya, Pinaki},
  booktitle={Journal of Physics: Conference Series},
  volume={405},
  pages={012017},
  year={2012},
  organization={IOP Publishing}
}

@article{dominguez2018introduction,
  title={{An introduction to isoparametric foliations}},
  author={Miguel Domínguez-Vázquez},
  journal={Preprint},
  pages={1--52},
  year={2018}
}


%%%%%%%%%%%%%%%%HHHHHHHHHHHHHHHHHHHHHHHHHHHHHHHHHHH%%%%%%%%%%%%%%%%%%%%%%%%%%%%%%%%%%%%%
@article{hartle72,
  title={Solutions of the {E}instein-{M}axwell equations with many black holes},
  author={Hartle, James B and Hawking, Stephen W},
  journal={Commun. Math. Phys.},
  volume={26},
  number={2},
  pages={87--101. MR1552582},
  year={1972},
  publisher={Springer}
}

@article{hamilton,
  title={Three-manifolds with positive {R}icci curvature},
  author={Hamilton, Richard S},
  journal={J. Diff. Geom.},
  volume={17},
  number={2},
  pages={255--306, MR0664497},
  year={1982},
  publisher={Lehigh University}
}


@misc{horowitz1992dark,
  title={The dark side of string theory: black holes and black strings},
  author={Horowitz, Gary T},
  journal={World Scientific},
  pages={81--06. MR1231344},
  year={1992},
  publisher={World Scientific}
}
@article{hollands2020quantum1,
  title={{Quantum instability of the Cauchy horizon in Reissner--Nordstr{\"o}m--deSitter spacetime}},
  author={Hollands, Stefan and Wald, Robert M and Zahn, Jochen},
  journal={Classical and Quantum Gravity},
  volume={37},
  number={11},
  pages={115009},
  year={2020},
  publisher={IOP Publishing}
}

@article{hollands2020quantum,
  title={{Quantum stress tensor at the Cauchy horizon of the Reissner--Nordstr{\"o}m--de Sitter spacetime}},
  author={Hollands, Stefan and Klein, Christiane and Zahn, Jochen},
  journal={Physical Review D},
  volume={102},
  number={8},
  pages={085004},
  year={2020},
  publisher={APS}
}


@article{hwang2021vacuum,
  title={{Vacuum static spaces with vanishing of complete divergence of Weyl tensor}},
  author={Hwang, Seungsu and Yun, Gabjin},
  journal={The Journal of Geometric Analysis},
  volume={31},
  number={3},
  pages={3060--3084},
  year={2021},
  publisher={Springer}
}


%%%%%%%%%%%%%%%%%%%%%JJJJJJJJJJJJJJJJJJJJJJJJJJJJ%%%%%%%%%%%%%%%%%%%%%%%%%%%%55
%\bibitem{jahns} Jahns, S. \emph{Photon sphere uniqueness in higher-dimensional electrovacuum space-times}. Classical Quantum Gravity 36 (2019), no. 23, 235019, 24 pp. MR4062728.

%%%%%%KKKKKKKKKKKKKKKKKKK%%%%%%%%%%%%%%%%%%%%%%%%%%%%%%%%%%%%%%%%%%%%%%%%%%%%%
%\bibitem{kobayashi} Kobayashi, O., Obata, M.: Conformally-Flatness and Static space-time. Manifolds and Lie Groups. Progress in Mathematics, vol. 14, pp. 197–206. Birkhuser, Boston (1981).

@article{kunduri2018,
  title={{No static bubbling spacetimes in higher dimensional Einstein--Maxwell theory}},
  author={Kunduri, Hari K and Lucietti, James},
  journal={Classical and Quantum Gravity},
  volume={35},
  number={5},
  pages={054003},
  year={2018},
  publisher={IOP Publishing}
}


@article{konoplya,
  title={{Instability of D-dimensional extremally charged Reissner-Nordstr{\"o}m (-de Sitter) black holes: extrapolation to arbitrary D}},
  author={Konoplya, RA and Zhidenko, A},
  journal={Physical Review D},
  volume={89},
  number={2},
  pages={024011},
  year={2014},
  publisher={APS}
}

%%%%%%%%%%%%%%%%%%%%%%%%%%LLLLLLLLLLLLLLLLLLLLLLLLL%%%%%%%%%%%%%%%%%%%%%%%%%%%%%%%%%%%%%

@article{benedito,
  title={Vanishing conditions on {W}eyl tensor for {E}instein-type manifolds},
  author={Leandro, Benedito},
  journal={Pacific J. Math.},
  volume={314},
  number={1},
  pages={99--113, MR4329972},
  year={2021},
  publisher={Mathematical Sciences Publishers}
}



%\bibitem{fernando} Leandro, B. and Coutinho, F. - \emph{Mean-stable surfaces in Static Einstein-Maxwell theory}, 2021. 	arXiv:2101.06142.

%\bibitem{anailton} Leandro, B., Melo, A. P., Menezes, I. and Pina, R. - \emph{Static Einstein-Maxwell space-time invariant by translation}, 2020. arXiv:2010.10708.

@article{leandro,
  title={Static perfect fluid spacetime with half conformally flat spatial factor},
  author={Leandro, Benedito and Sol{\'o}rzano, Newton},
  journal={Manuscripta Math.},
  volume={160},
  number={1},
  pages={51--63, MR3983386},
  year={2019},
  publisher={Springer}
}

@inproceedings{lucietti,
  title={All Higher-Dimensional {M}ajumdar--{P}apapetrou {B}lack {H}oles},
  author={Lucietti, James},
  booktitle={Ann. Henri Poincar{\'e}},
  pages={1--14. MR4285947},
  year={2021},
  organization={Springer}
}


%%%%%%%%%%%%%%%%%%%%%%%%%%%%%%%MMMMMMMMMMMMMMMMMMMMMMMMMMMMMMMMMM%%%%%%%%%%%%%%%%%%%%%%%%%%%%%%%%%%%%%%%%%%%%%%




%\bibitem{mtw} Misner, C. W.; Thorne, K. S.;  Wheeler, J. A. \emph{Grazitation}. Freeman, San Francisco, 1973. 


%%%%%%%%%%%%%%%%%%%%%%%%%%%%%%%%%%%OOOOOOOOOOOOOOOOOOOOOOOOOOOOOOOOOOOOOOOOOOOOOOOOOOOOOOOOOOOOOOOOOOO%%%%%%%%%%%%%%%%%%%%%%%%%%%%%%%%%
@book{oniell,
  title={Semi-Riemannian geometry with applications to relativity},
  author={O'neill, Barrett},
  year={1983},
  publisher={Academic press}
}

%%%%%%%%%%%%%%QQQQQQQQQQQQQQQQQQQQQQQQQqq%%%%%%%%%%%%%%%%%%%%%%%%%%%%%%%%%%%%%%%%

%\bibitem{qing} Qing, J.; Yuan, W. \emph{A note on static spaces and related problems}. J. Geom. Phys. 74 (2013): 18-27. MR3118569.
@article{qing2013note,
  title={A note on static spaces and related problems},
  author={Qing, Jie and Yuan, Wei},
  journal={Journal of Geometry and Physics},
  volume={74},
  pages={18--27},
  year={2013},
  publisher={Elsevier}
}

%%%%%%%%%%RRRRRRRRRRRRRRRRRRRRRRRRRRRRRR%%%%%%%%%%%%%%%%%%%%%%%%%%%%%%%%%

%\bibitem{ruback1988}{Ruback, P.}  {\em A new uniqueness theorem for charged black holes.} Classical Quantum Gravity 5 (1988), no. 10, L155–L159. MR0964972.

%%%%%%%%%%%%%%%%%%%%%%%SSSSSSSSSSSSSSSSSSSSSSSSSSSSSSSSSSSSSSSSSSSSSSS%%%%%%%%%%%%%%%%%%%%%%%%%%%%%%%%%%%%%%

%\bibitem{JP}  dos Santos, J. P.; Leandro, B. \emph{Reduction of the n-dimensional static vacuum Einstein equation and generalized Schwarzschild solutions}. J. Math. Anal. Appl. 469 (2019), no. 2, 882–896. MR3860452.
@book{scorpan,
  title={The wild world of 4-manifolds},
  author={Scorpan, Alexandru},
  year={2005, MR2136212},
  publisher={American Mathematical Society}
}

@book{synge,
  title={Relativity: the general theory},
  author={Synge, John Lighton},
  year={1960. MR0118457},
  publisher={North-Holland, Amsterdam}
}

@article{szekeres1968conformal,
  title={Conformal tensors},
  author={Szekeres, Peter},
  journal={Proceedings of the Royal Society of London. Series A. Mathematical and Physical Sciences},
  volume={304},
  number={1476},
  pages={113--122},
  year={1968},
  publisher={The Royal Society London}
}


 %\bibitem{schoen} {Schoen, R. \& Yau, S-T.} - {\em Positive Scalar Curvature and Minimal Hypersurface Singularities.} arxiv:1704.05490, April 2017.

%%%%%%%%%%%%%%%WWWWWWWWWWWWWWWWWWWWWWWWWWWWWWWWWWWWWWWWWWWWWWWWWWWWWWWWWWW$$$$$$$$$$$$$$$$$$$$$$$$$$$$$$$$$$$$$$$$$$$$$$$$$$$$$4

@article{weber,
  title={A gap theorem for half-conformally flat manifolds},
  author={Weber, Brian and Citoler-Saumell, Martin},
  journal={Illinois J. Math.},
  volume={65},
  number={1},
  pages={71--96. MR4278679},
  year={2021},
  publisher={Duke University Press}
}

@article{wu,
  title={{A Weitzenböck formula for canonical metrics on four-manifolds}},
  author={Wu, Peng},
  journal={Trans. Amer. Math. Soc.},
  volume={369},
  number={2},
  pages={1079--1096, MR3572265},
  year={2017}
}



%%%%%%%%%%%%%%%%%%%%%%%%%%%%%%%%%%%%%%%%%%%%%%%YYYYYYYYYYYYYYYYYYYYYYYYYYYYYYYYYYYYY%%%%%%%%%%%%%%%%%%%%%%%%%%%%%




 %%%%%%%%%%%%%%%%%%%%%%%%%%%%%%%%%%%%%ZZZZZZZZZZZZZZZZZZZZZZZZZZZZZZZZZZZZZZZZZZZZZZZZZ%%%%%%%%%%%%%%%%%%%%%%%%%%%%%%%%%%%%%%%%%%

@article{zhu,
  title={The classification of complete locally conformally flat manifolds of nonnegative {R}icci curvature},
  author={Zhu, Shun-Hui},
  journal={Pacific J. Math.},
  volume={163},
  number={1},
  pages={189--199},
  year={1994},
  publisher={Mathematical Sciences Publishers}
}



\begin{thebibliography}{1}
    
\bibitem{andrade} Maria Andrade, Benedito Leandro, and Róbson Lousa, \emph{On the geometry of electrovacuum spaces in higher dimensions}, arXiv preprint arXiv:2107.14605 (2021), 1–23.

\bibitem{besse} Arthur L Besse, \emph{Einstein manifolds}, Springer Science \& Business Media, 2007, MR2371700.

\bibitem{cao} Huai-Dong Cao and Qiang Chen, \emph{On Bach-flat gradient shrinking Ricci solitons}, Duke Math. J. 162 (2013), no. 6, 1149–1169, MR3053567.

\bibitem{sun} Huai-Dong Cao, Xiaofeng Sun, and Yingying Zhang, \emph{On the structure of gradient Yamabe solitons}, Math. Res. Lett. 19 (2012), no. 4, 767–774, MR3008413.

\bibitem{catino2017gradient} Giovanni Catino, Paolo Mastrolia, and Dario Daniele Monticelli, \emph{Gradient Ricci solitons with vanishing conditions on Weyl}, Journal De Mathématiques Pures Et Appliquées 108 (2017), no. 1, 1–13, MR3660766.

\bibitem{cederbaum2016uniqueness} Carla Cederbaum and Gregory J Galloway, \emph{Uniqueness of photon spheres in electro-vacuum spacetimes}, Classical and Quantum Gravity 33 (2016), no. 7, 075006, MR3471730.

\bibitem{chow2006hamilton} Bennett Chow, Peng Lu, and Lei Ni, \emph{Hamilton’s Ricci flow}, vol. 77, American Mathematical Soc., 2006, MR2274812.

\bibitem{chrusciel2005non} Piotr T. Chrusciel and Erwann Delay, \emph{Non-singular, vacuum, stationary space-times with a negative cosmological constant}, Annales Henri Poincare 8 (2007), 219–239, MR2314449.

\bibitem{chrusciel2017non} Piotr T. Chrusciel and Erwann Delay, \emph{Non-singular space-times with a negative cosmological constant: II. Static solutions of the Einstein–Maxwell equations}, Lett. Math. Phys. 107 (2017), no. 8, 1391–1407, MR3669238.

\bibitem{tiarlos} Tiarlos Cruz, Vanderson Lima, and Alexandre de Sousa, \emph{Min-max minimal surfaces, horizons and electrostatic systems}, to appear in Jour. Diff. Geom., (arXiv:1912.08600), 1–47.

\bibitem{hwang2021vacuum} Seungsu Hwang and Gabjin Yun, \emph{Vacuum static spaces with vanishing of complete divergence of Weyl tensor}, The Journal of Geometric Analysis 31 (2021), no. 3, 3060–3084, MR4225834.

\bibitem{kunduri2018} Hari K Kunduri and James Lucietti, \emph{No static bubbling spacetimes in higher dimensional Einstein–Maxwell theory}, Classical and Quantum Gravity 35 (2018), no. 5, 054003.

\bibitem{benedito} Benedito Leandro, \emph{Vanishing conditions on Weyl tensor for Einstein-type manifolds}, Pacific J. Math. 314 (2021), no. 1, 99–113, MR4329972.

\bibitem{lucietti} James Lucietti, \emph{All higher-dimensional Majumdar–Papapetrou Black Holes}, Ann. Henri Poincaré, Springer, 2021, pp. 1–14. MR4285947.

\bibitem{qing2013note} Jie Qing and Wei Yuan, \emph{A note on static spaces and related problems}, Journal of Geometry and Physics 74 (2013), 18–27, MR3118569.

\bibitem{szekeres1968conformal} Peter Szekeres, \emph{Conformal tensors}, Proceedings of the Royal Society of London. Series A. Mathematical and Physical Sciences 304 (1968), no. 1476, 113–122.
\end{thebibliography}


\end{document}